\documentclass[letterpaper]{article}

\usepackage{amsmath}
\usepackage{amsfonts}
\usepackage{amsthm}
\usepackage{amssymb}
\usepackage{datetime}
\usepackage{color}

\theoremstyle{plain}
\newtheorem{theorem}{Theorem}[section]
\newtheorem{lemma}[theorem]{Lemma}
\newtheorem{corollary}[theorem]{Corollary}
\theoremstyle{definition}
\newtheorem{definition}[theorem]{Definition}
\newtheorem{problem}[theorem]{Problem}
\newtheorem{example}[theorem]{Example}
\newtheorem{remark}[theorem]{Remark}

\numberwithin{equation}{section}

\newcommand{\all}{\hbox{for all}}

\newcommand{\bra}[2]{\langle#1,#2\rangle}

\newcommand{\bigcupn}{\bigcup\nolimits}
\newcommand{\Bra}[2]{\big\langle#1,#2\big\rangle}

\newcommand{\dbs}{^{**}}

\newcommand{\dom}{\hbox{\rm dom}}

\newcommand{\eps}{\varepsilon}

\newcommand{\F}{{\mathbb F}}

\newcommand{\half}{{\textstyle\frac{1}{2}}}

\newcommand{\ifff}{\Longleftrightarrow}
\newcommand{\infn}{\inf\nolimits}
\newcommand{\intr}{\hbox{\rm int}}

\newcommand{\lr}{\Longrightarrow}

\newcommand{\PCLSC}{{\cal PCLSC}}

\newcommand{\qlr}{\quad\Longrightarrow\quad}

\newcommand{\quand}{\quad\hbox{and}\quad}
\newcommand{\rbar}{\,]{-}\infty,\infty]}

\newcommand{\RR}{\mathbb R}
\newcommand{\rl}{\Longleftarrow}

\newcommand{\supn}{\sup\nolimits}

\newcommand{\toto}{\rightrightarrows}

\newcommand{\UT}{{\wt U}}

\newcommand{\wh}{\widehat}
\newcommand{\wt}{\widetilde}
\newcommand{\WT}{{\wt W}}

\newcommand{\Def}{Definition~\ref}

\newcommand{\Ex}{Example~\ref}

\newcommand{\Lem}{Lemma~\ref}

\newcommand{\Prob}{Problem~\ref}
\newcommand{\Probs}{Problems~\ref}

\newcommand{\Sec}{Section~\ref}

\newcommand{\Thm}{Theorem~\ref}
\newcommand{\Thms}{Theorems~\ref}

\title{Quasidensity: a survey and some examples}
\author{
Stephen Simons
\thanks{
Department of Mathematics, University of California, Santa Barbara, CA\ 93106-3080, U.S.A.
Email: \texttt{stesim38@gmail.com}.}}

\date{}

\begin{document}

\maketitle

\begin{abstract}\noindent
In three previous papers, we discussed quasidense multifunctions from a Banach space into its dual, or, equivalently, quasidense subsets of the product of a Banach space and its dual.  In this paper, we survey (without proofs) some of the main results about quasidensity, and give some simple limiting examples in Hilbert spaces, reflexive Banach spaces, and nonreflexive Banach spaces.
\end{abstract}




\section{Introduction}
This is a sequel to the papers \cite{PARTTWO} and \cite{SW},  in which we discussed {\em quasidense} multifunctions from a Banach space into its dual.   A number of the results in \cite{PARTTWO} depend on the somewhat more abstract analysis that appears in \cite{PARTONE}.
\par
In \Sec{BANsec}, we give some Banach space notation and definitions.
\par
Let $S$ be a multifunction (not assumed to be monotone) from a Banach space into its dual.   We define the {\em quasidensity} of $S$ in \Def{QDdef}.   In \Thm{SWthm}, we establish that the (appropriately defined) subdifferential of a proper (not necessarily convex) lower semicontinuous function is quasidense, and we show in the simple \Ex{RCASE} that the condition \eqref{SW1}, which is sufficient for the quasidensity, is not necessary.   
\par
In \Sec{MONsec}, we start our investigation of {\em monotone} multifunctions and collect together some of the results that were proved in \cite{PARTTWO}, with references to the original proofs in \cite{PARTTWO} or \cite{PARTONE}, as the case may be.   We point out in \Thm{RLMAXthm} and \Ex{TAILex} that every closed, monotone quasidense multifunction is maximally monotone, but that there exist maximally monotone linear operators that are not quasidense.   We point out in \Thm{RTRthm} that the subdifferental of any proper, convex lower semicontinuous function is quasidense.   By virtue of \Thm{RLMAXthm}, \Thm{RTRthm} generalizes Rockafellar's result that such subdifferentials are maximally monotone.   In \Thm{STDthm} we prove that the sum of a pair of closed, monotone quasidense multifunctions that satisfies the Rockafellar constraint qualification is closed, monotone and quasidense.   We note that it is apparently not known whether the sum of a pair of maximally monotone multifunctions that satisfies the Rockafellar constraint qualification is necessarily maximally monotone.   (This is known as the {\em sum problem}.)   In \Thm{STRthm} we give a ``parallel'' sum theorem for a pair of closed, monotone quasidense multifunctions that satisfy the ``dual'' of the Rockafellar constraint qualification.   In the process of doing this we introduce the {\em Fitzpatrick function} and {\em Fitzpatrick extension} of a closed, monotone, quasidense multifunction.   In \Probs{AFMAXprob} and \ref{FITZDIFFprob} we give two questions that merit further study.   
\par
Quasidensity has connections with many of the subclasses of the maximally monotone multifunctions that have been investigated over the years.   We explore just three of these in \Sec{CLASSsec}: {\em type (FPV)}, {\em type (FP)} and {\em strongly maximal}.   \Probs{FPVprob}, \ref{FPprob} and \ref{STRONGprob} contain open questions about these three subclasses of the maximally monotone multifunctions.   Other related subclasses are discussed in \cite[Theorem 8.1]{PARTTWO}, \cite[Theorem 8.2]{PARTTWO}, \cite[Theorem 11.6, p.\ 1045]{PARTONE} and \cite[Theorem 11.9, pp.\ 1045--1046]{PARTONE}.
\par
In the final three sections, we show how quasidensity behaves in three special cases: Hilbert spaces in \Sec{Hsec}, reflexive Banach spaces in \Sec{Rsec} and nonreflexive Banach spaces in \Sec{NRsec}.
\par
The author would like to express his thanks to Hedy Attouch and Heinz Bauschke for constructive discussions about the topics discussed in this paper.   He would also like to thank Xianfu Wang for constructive comments about an earlier version of this paper. 
\section{Banach space notation and definitions}\label{BANsec}
If $X$ is a nonzero real Banach space and $f\colon\ X \to \rbar$, we write $\dom\,f$ for the set $\big\{x \in X\colon\ f(x) \in \RR\big\}$.   $\dom\,f$ is the {\em effective domain} of $f$.   We say that $f$ is {\em proper} if $\dom\,f \ne \emptyset$.   We write $\PCLSC(X)$ for the set of all proper convex lower semicontinuous functions from $X$ into $\rbar$.   We write $X^*$ for the dual space of $X$ \big(with the pairing $\bra\cdot\cdot\colon X \times X^* \to \RR$\big).  If $f \in \PCLSC(X)$ then, as usual, we define the {\em Fenchel conjugate}, $f^*$, of $f$ to be the function on $X^*$ given by $x^* \mapsto \supn_X\big[x^* - f\big]$.
We write $X\dbs$ for the bidual of $X$ \big(with the pairing $\bra\cdot\cdot\colon X^* \times X\dbs \to \RR$\big).   If $f \in \PCLSC(X)$ and $f^* \in \PCLSC(X^*)$, we define $f\dbs\colon X\dbs \to \rbar$ by $f\dbs(x\dbs) := \sup_{X^*}\big[x\dbs - f^*\big]$.   If $x \in X$, we write $\wh x$ for the canonical image of $x$ in $X\dbs$, that is to say, for all $(x,x^*) \in X \times X^*$,\break $
\bra{x^*}{\wh x} = \bra{x}{x^*}$.\quad If $f \in \PCLSC(X)$ then the {\em convex subdifferential of} $f$ is the multifunction $\partial f\colon\ E  \toto E^*$ that satisfies
$$x^* \in \partial f(x) \ifff f(x) + f^*(x^*) = \bra{x}{x^*}.$$
\par
We suppose that $E$ is a nonzero real Banach space with dual $E^*$.   For all $(x,x^*) \in E \times E^*$, we write $\|(x,x^*)\| := \sqrt{\|x\|^2 + \|x^*\|^2}$.   We represent $(E \times E^*)^*$ by $E^* \times E\dbs$, under the pairing
$$\Bra{(x,x^*)}{(y^*,y\dbs)} := \bra{x}{y^*} + \bra{x^*}{y\dbs}.$$
The dual norm on $E^* \times E\dbs$ is given by  $\|(y^*,y\dbs)\| := \sqrt{\|y^*\|^2 + \|y\dbs\|^2}$.
\smallbreak
Now let $S\colon\ E \toto E^*$.   We write $G(S)$ for the graph of $S$, $D(S)$ for the domain of $S$ and $R(S)$ for the range of $S$.
We will always suppose that $G(S) \ne \emptyset$ (equivalently, $D(S) \ne \emptyset$ or $R(S) \ne \emptyset$).   We say that $S$ is {\em closed} if $G(S)$ is closed.   If $x \in E$, we define the multifunction $_xS\colon E \toto E^*$ by $_xS = (S^{-1} - x)^{-1}$. Then ${_x}S(t) = S(t + x)$.   We write $J\colon E \toto E^*$ for the {\em duality map}.   We recall that $J$ is maximally monotone and
\begin{equation}\label{J1}
x^* \in Jx \iff \half\|x\|^2 + \half\|x^*\|^2 = \bra{x}{x^*} \iff \|x\|^2 = \|x^*\|^2 = \bra{x}{x^*}.
\end{equation}
\section{Quasidensity}\label{QDsec}    
\begin{definition}\label{QDdef}
We say that $S$ is {\em quasidense} if, for all $(x,x^*) \in E \times E^*$,
\begin{equation*}
\infn_{(s,s^*) \in G(S)}\big[\half\|s - x\|^2 + \half\|s^* - x^*\|^2 + \bra{s - x}{s^* - x^*}\big] \le 0.
\end{equation*}
See \cite[Definition 3.1]{PARTTWO} and \cite[Example 7.1, eqn.\ (28), p.\ 1031)]{PARTONE}.
\end{definition}
We have the following simple result connecting $J$ and quasidensity:
\begin{lemma}\label{Jlem}
Let $S\colon\ E \toto E^*$ and, for all $x \in E$, $_xS + J$ be surjective.   Then $S$ is quasidense. 
\end{lemma}
\begin{proof}
Let $(x,x^*) \in E \times E^*$.   Choose $t \in D(_xS)$ such that $(_xS + J)t = x^*$.   So there exists $s^* \in S(t + x)$ such that $(t,x^* - s^*) \in G(J)$.   Thus, writing 
$s := t + x$, $(s,s^*) \in G(S)$ and $(s - x,x^* - s^*) = (t,x^* - s^*) \in G(J)$, that is to say $\half\|s - x\|^2 + \half\|x^* - s^*\|^2 = \bra{s - x}{x^* - s^*}$.   Equivalently,
\par
\centerline{$\half\|s - x\|^2 + \half\|s^* - x^*\|^2 + \bra{s - x}{s^* - x^*} = 0$.}
\noindent
This obviously implies that $S$ is quasidense.
\end{proof}  
We now discuss a significant example of quasidensity.   The following definition was made in \cite[Definition 2.1, p.\ 633]{SW}. 
\begin{definition}\label{WSUBDIFF}
An {\em ubiquitous subdifferential}, $\partial_u$, is a rule that associates with each proper lower semicontinuous function $f\colon\ E\ \to \rbar$ a multifunction $\partial_u f: E \toto E^*$ such that\par
$\bullet$\enspace  $\partial_u f(x) = \emptyset$ if $x \not\in \dom\,f$,\par
$\bullet$\enspace $0\in\partial_u f(x)$ if $f$ attains a strict global minimum at $x$,\par
$\bullet$\enspace $\partial_u (f+h)(x) \subseteq \partial_u f(x) + \partial h(x)$ whenever $x \in \dom\,f$ and $h$ is a continuous convex real function on $E$ (here $\partial h$ is the convex subdifferential of $h$).
\end{definition}
\par
There is a list of abstract subdifferentials that satisfy these conditions in the remarks following \cite[Definition 2.1]{SW}.   Now suppose that $\partial_u$ is an ubiquitous subdifferential.  We have the following result:
\begin{theorem}\label{SWthm}
Let $f:E \to \RR$ be proper and lower semicontinuous.   Let $a_0,b_0,c_0 \in \RR$ with $a_0 < \half$ and,
\begin{equation}\label{SW1}
\all\ x \in E,\ f(x) \ge -a_0\|x\|^2 - b_0\|x\| - c_0.
\end{equation}
Then $\partial_u f$ is quasidense.
\end{theorem}
\begin{proof}
See \cite[Theorem 3.2, pp.\ 634--635]{SW}.   The proof of this is based on the ``elementary'' proof of \Thm{RTRthm}, that is \cite[Theorem 4.6]{PARTTWO}.
\end{proof}
\begin{corollary}\label{SWcor}
Let $f:E\rightarrow\rbar$ be proper, lower semicontinuous and dominate a continuous affine function.  Then $\partial_u f$ is quasidense.
\end{corollary}
\begin{proof}
This is immediate from \Thm{SWthm}.
\end{proof}
\begin{example}\label{RCASE}
In this example, we suppose that $E = \RR$ and that $\partial_u$ has the special property that, whenever $f$ is a polynomial, $\partial_uf(x) = \{f^\prime(x)\}$.   For instance, $\partial_u$ could be the Clarke--Rockafellar subdifferential.   Let $f$ be a polynomial.  Then the statement that $\partial_u f$ is quasidense can be rewritten:
\begin{equation*}
\all\ z \in \RR,\ \infn_{s \in \RR}\half(s + f^\prime(s) - z)^2 \le 0.
\end{equation*}
Let $\lambda \in \RR$ and $f(x) := -\lambda x^2$.   So $\partial_u f$ is quasidense if, and only if,
\begin{equation*}
\all\ z \in \RR,\ \infn_{s \in \RR}\half(s - 2\lambda s - z)^2 \le 0.
\end{equation*}
$\bullet$\enspace If $\lambda \ne \half$ then taking $s := z/(1 - 2\lambda)$ shows that $\partial_u f$ is quasidense.
\par\noindent
$\bullet$\enspace If $\lambda = \half$ then taking $z \ne 0$ shows that $\partial_u f$ is not quasidense.
\par
Thus the condition \eqref{SW1} is sufficient but not necessary for the quasidensity of $\partial_u f$.
\par
This example (with different justification) is taken from \cite[Example 3.5, p.\ 636]{SW}.
\end{example}
\section{Monotone multifunctions: basic results}\label{MONsec}
For the rest of this paper, we will discuss the very rich theory of the quasidensity of {\em monotone} multifunctions.
\begin{theorem}[Quasidensity and maximality]\label{RLMAXthm}
Let $S\colon\ E \toto E^*$ be closed, monotone and quasidense.   Then $S$ is maximally monotone. 
\end{theorem}
\begin{proof}
See \cite[Theorem 3.2]{PARTTWO}, \cite[Theorem 7.4(a), pp.\ 1032--1033]{PARTONE} or \cite[Lemma 4.7, p.\ 1027]{PARTONE}.    
\end{proof}
\begin{example}[The tail operator]\label{TAILex}
Let $E = \ell_1$, and define the linear map $T\colon\ \ell_1 \mapsto \ell_\infty = E^*$ by $(Tx)_n = \sum_{k \ge n} x_k$.   $T$ is maximally monotone but not quasidense.   See \cite[Example 7.10, pp.\ 1034--1035]{PARTONE}. 
\end{example}
\par
\Thm{RTRthm} below is a very important result.   By virtue of \Thm{RLMAXthm}, it generalizes Rockafellar's result, \cite{RTRMMT}, that subdifferentials of proper, convex, lower semicontinuous functions are maximally monotone.   The first proof of this result mentioned below was the source of \Thm{SWthm}.
\begin{theorem}\label{RTRthm}
Let $f \in \PCLSC(E)$.   Then  $\partial f$ is closed, monotone and\break quasidense.
\end{theorem}
\begin{proof}
The more elementary proof of this result (see \cite[Theorem 4.6]{PARTTWO}) uses the Br{\o}ndsted--Rockafellar theorem \cite{BRON} and  Rockafellar's formula \cite{FENCHEL} for the subdifferential of a sum.   There is a slicker but more sophisticated proof using the properties of {\em Fitzpatrick functions} (see below) in \cite[Theorem 7.5, p.\ 1033]{PARTONE}. 
\end{proof}
As we noted in the introduction, it is apparently not known whether the result corresponding to \Thm{STDthm} with ``closed, monotone and quasidense'' replaced by ``maximally monotone'' is true.
\begin{theorem}[Sum theorem with domain constraints]\label{STDthm}
Let $S,T\colon\ E \toto E^*$ be closed, monotone and quasidense and $D(S) \cap \intr\,D(T) \ne \emptyset$.   Then $S + T$ is closed, monotone and quasidense.
\end{theorem}
\begin{proof}
This is a special case of \cite[Theorem 8.4(a)$\lr$(d), pp.\ 1036--1037]{PARTONE}.
\end{proof}
There is a ``dual'' version of \Thm{STDthm} that we will state in \Thm{STRthm}.   Before discussing this, we introduce the {\em Fitzpatrick function}, $\varphi_S\colon\ E \times E^* \to \rbar$, and the {\em Fitzpatrick extension}, $S^\F\colon\ E^* \toto E\dbs$, of a closed, monotone, quasidense multifunction $S\colon\ E \toto E^*$.   The function $\varphi_S$ is defined by
\begin{equation*}
\varphi_S(x,x^*) := \supn_{(s,s^*) \in G(S)}\big[\bra{s}{x^*} + \bra{x}{s^*} - \bra{s}{s^*}\big].
\end{equation*}
See \cite{FITZ}, \cite[Definition 3.4]{PARTTWO} and many other places.   The multifunction $S^\F$ was defined in \cite[Definition 5.1]{PARTTWO} by
\begin{equation*}
(y^*,y\dbs) \in G(S^\F) \hbox{ exactly when } {\varphi_S}^*(y^*,y\dbs) = \bra{y^*}{y\dbs}.
\end{equation*}
(There is a more abstract version of this in  \cite[Definition 8.5, p.\ 1037]{PARTONE}.)   The word {\em extension} is justified by the easily verifiable fact that
\begin{equation*}
(x,x^*) \in G(S) \iff (x^*,\wh x) \in G(S^\F).
\end{equation*}
It was shown in \cite[Section 11]{PARTTWO} that $(y^*,y\dbs) \in G(S^\F)$ exactly when $(y\dbs,y^*)$ is in the {\em Gossez extension} of $G(S)$ \big(see \cite[Lemma~2.1, p.\ 275]{GOSSEZ}\big).
\begin{theorem}[Sum theorem with range constraints]\label{STRthm}
Let $S,T\colon\ E \toto E^*$ be closed, monotone and quasidense and $R(S) \cap \intr\,R(T) \ne \emptyset$. Then the multifunction $y \mapsto (S^\F + T^\F)^{-1}(\wh y)$ is closed, monotone and quasidense.
Under certain additional technical conditions, the {\em parallel sum} $(S^{-1} + T^{-1})^{-1}$ is closed, monotone and  quasidense.
\end{theorem}
\begin{proof}
This is a special case of \cite[Theorem 8.8, p.\ 1039]{PARTONE}.
\end{proof}
If $S\colon\ E \toto E^*$ is closed, monotone and quasidense then it is easily seen that $S^\F$ is monotone.   In fact, we have the following stronger nontrivial result: 
\begin{theorem}\label{AFMAXthm}
Let $S\colon\ E \toto E^*$ be closed, monotone and quasidense.   Then $S^\F\colon E^* \toto E\dbs$ is maximally monotone.
\end{theorem}
\begin{proof}
See \cite[Lemma 12.5, p.\ 1047]{PARTONE}.   There is also a sketch of a proof in \cite[Section 11]{PARTTWO}.
\end{proof}
This leads to the following problem:
\begin{problem}\label{AFMAXprob}
Let $S\colon\ E \toto E^*$ be closed, monotone and quasidense.   Then is $S^\F\colon E^* \toto E\dbs$ necessarily quasidense?
\end{problem}   
\begin{theorem}\label{FITZDIFFthm}
Let $f \in \PCLSC(E)$.   Then $(\partial f)^\F = \partial(f^*)$.
\end{theorem}
\begin{proof}
See \cite[Theorem 5.7]{PARTTWO}.
\end{proof}
\begin{remark}\label{FITZDIFFrem}
\Thm{FITZDIFFthm} is equivalent to \cite[Th\'eor\`eme~3.1, pp.\ 376--378]{GOSSEZ}.
\end{remark}
\begin{problem}\label{FITZDIFFprob}
The proof of \cite[Theorem 5.7]{PARTTWO} (invoked in \Thm{FITZDIFFthm}) is quite convoluted.   Is there a simple direct proof of this result?
\end{problem}
\begin{remark}\label{FGrem}
\Thms{FITZDIFFthm} and \ref{RTRthm} show that {\em if $f \in \PCLSC(E)$ then $(\partial f)^\F$ is quasidense}, in other words, in this restricted situation we have a positive solution to \Prob{AFMAXprob}.   
\end{remark}
\section{Quasidensity and the classification of\\ maximally monotone multifunctions}\label{CLASSsec}
The closed, monotone, quasidense multifunctions have relationships with many other subclasses of the maximally monotone multifunctions.   We shall discuss just three of these.   Four others are mentioned in the introduction.
\begin{definition}\label{FPVdef}
Let $S\colon\ E \toto E^*$ be monotone.   We say that $S$ is {\em of type (FPV)} or {\em maximally monotone locally} if, whenever $U$ is an open convex subset of $E$, $U \cap D(S) \neq \emptyset$, $(w,w^*) \in U \times E^*$ and
\begin{equation}\label{FPV1}
(s,s^*) \in G(S)\quand s \in U \qlr \bra{s - w}{s^* - w^*}\ge 0,
\end{equation}
then $(w,w^*) \in G(S)$.   (If we take $U = E$, we see that every monotone\break multifunction of type (FPV) is maximally monotone.)  See \cite[pp.\ 150--151]{HBM}. 
\end{definition}
\begin{theorem}\label{FPVthm}
Any closed, monotone, quasidense multifunction is maximally monotone of type (FPV).
\end{theorem}
\begin{proof}
See \cite[Theorem 7.2]{PARTTWO}. 
\end{proof}
\begin{problem}\label{FPVprob}
Is every maximally monotone multifunction of type (FPV)?    The tail operator (see \Ex{TAILex}) does not provide an negative example because it was proved in  Fitzpatrick--Phelps, \cite[Theorem 3.10, p.\ 68]{FITZTWO} that if $S\colon\ E \toto E^*$ is maximally monotone and $D(S) = E$ then $S$ is of type (FPV).      Also, it was proved in \cite[Theorem 46.1, pp.\ 180--182]{HBM} that {\em if $S\colon E \toto E^*$ is maximally monotone and $G(S)$ is convex then $S$ is of type (FPV).}   A negative example would lead to a negative solution for the sum problem.   See \cite[Theorem 44.1, p.\ 170]{HBM} 
\end{problem}
\begin{definition}\label{FPdef}
Let $S\colon\ E \toto E^*$ be monotone.   We say that $S$ is {\em of type (FP)} or {\em locally maximally monotone} if, whenever $\UT$ is a convex open subset of $E^*$, $\UT \cap R(S) \neq \emptyset$, $(w,w^*) \in E \times \UT$ and
\begin{equation}\label{FP41}
(s,s^*) \in G(S)\ \hbox{and}\ s^* \in \UT \qlr \bra{s - w}{s^* - w^*} \ge 0,
\end{equation}
then $(w,w^*) \in G(S)$.  \big(If we take $\UT = E^*$, we see that every monotone multifunction of type (FP) is maximally monotone.\big) See \cite[pp.\ 149--150]{HBM}.
\end{definition}
\begin{theorem}\label{FPthm}
A maximally monotone multifunction is quasidense $\iff$ it is of type (FP).  
\end{theorem}
\begin{proof}
See \cite[Theorem 10.3]{PARTTWO}. This result is related to results of Marques Alves and Svaiter, \cite[Theorem~1.2(1$\ifff$5), p.\ 885]{ASD}, Voisei and Z\u{a}linescu, \cite[Theorem 4.1, pp.\ 1027--1028]{VZ} and Bauschke, Borwein, Wang and Yao, \cite[Theorem~3.1, pp.\ 1878--1879]{BBWYFP}.   
\end{proof}
\begin{problem}\label{FPprob}
The proof of ($\lr$) in \Thm{FPthm} relies on \cite[Lemma 10.1]{PARTTWO}.   Is there a {\em simple direct} proof of this result?   In this connection, see also the proof of \cite[Lemma 12.2]{PARTTWO}, which is hardly simple and direct.
\end{problem}
\begin{definition}\label{STRONGdef}
Let $S\colon\ E \toto E^*$ be monotone.  We say that $S$ is {\em strongly maximal} (see \cite[Theorems~6.1-2, pp.\ 1386--1387]{CONTROLLED}) if, whenever $w \in E$ and $\WT$ is a nonempty $w(E^*,E)$--compact convex subset of $E^*$ such that, 
\begin{equation*}
\all\ (s,s^*) \in G(S),\ \max\bra{s - w}{s^* - \WT} \ge 0,\end{equation*}
then $Sw \cap \WT \ne \emptyset$ and, further, whenever $W$ is a nonempty $w(E,E^*)$--compact convex subset of $E$, $w^* \in E^*$ and,
\begin{equation*}
\all\ (s,s^*) \in G(S),\ \max\bra{s - W}{s^* - w^*} \ge 0,
\end{equation*}
then $w^* \in S(W)$.   This property was originally proved for convex subdifferentials.   If $S$ is strongly maximal then clearly $S$ is maximal.
\end{definition}
\begin{theorem}\label{STRONGthm}
Let $S\colon\ E \toto E^*$ be closed, monotone and quasidense.  Then $S$ is strongly maximal.
\end{theorem}
\begin{proof}
See \cite[Theorem 8.5]{PARTTWO}.
\end{proof}
\begin{problem}\label{STRONGprob}
\par
Is every maximally monotone multifunction strongly maximal?   The tail operator (see \Ex{TAILex}) does not provide a negative example because it was proved in  Bauschke--Simons, \cite[Theorem 1.1, pp.\ 166--167]{BS} that if $S\colon\ D(S) \subset E \to E^*$ is linear and maximally monotone then $S$ is strongly maximal.   More generally, it was proved in \cite[Theorem 46.1, pp.\ 180--182]{HBM} that {\em if $S\colon E \toto E^*$ is maximally monotone and $G(S)$ is convex then $S$ is strongly maximal}. 
\end{problem}
\section{The Hilbert space case}\label{Hsec}
Let $H$ be a real Hilbert space and $I\colon\ H \to H$ be the identity map. As usual, we identify $H^*$ with $H$.   Let $S\colon\ H \toto H$.    From \Def{QDdef} and the properties of Hilbert spaces, $S$ is quasidense exactly when, for all $(x,x^*) \in H \times H$,
\begin{equation*}
\infn_{(s,s^*) \in G(S)}\half\|s + s^* - x - x^*\|^2 \le 0,
\end{equation*}
that is to say, for all $z^* \in H$, $\infn_{(s,s^*) \in G(S)}\half\|s + s^* - z^*\|^2 \le 0$.   This is equivalent to the statement that
$\{s + s^*\colon\ (s,s^*) \in G(S)\}$ is dense in $H$, that is to say $R(S + I)$ is dense in $H$.   This leads to the following result:
\begin{theorem}\label{Hthm}
Let $S\colon\ H \toto H$ be closed and monotone.   Then $S$ is quasidense if, and only if, $S + I$ is surjective. 
\end{theorem}
\begin{proof}
``If'' is obvious from the comments above.   Suppose, conversely, that $S$ is quasidense.   Then, from \Thm{RLMAXthm}, $S$ is maximally monotone, and the surjectivity of $S + I$ follows from Minty's theorem.
\end{proof}
Monotonicity plays a mysterious role in \Thm{Hthm}.   This is shown by the following example.
\begin{example}
Define $S\colon\ \RR \toto \RR$ by
\begin{equation*}
S(x) :=
\begin{cases}
\{1/x - x\}&(x \ne 0);\\
\emptyset&(x = 0).
\end{cases}
\end{equation*}
Clearly, $S$ is closed.  Then
\begin{equation*}
(S + I)(x) =
\begin{cases}
\{1/x\}&(x \ne 0);\\
\emptyset&(x = 0).
\end{cases}
\end{equation*}
Thus $R(S + I) = \RR \setminus \{0\}$.   Since this is dense in $\RR$, $S$ is quasidense.   But $S + I$ is manifestly not surjective. 
\end{example}
\section{The reflexive Banach space case}\label{Rsec}
Let $E$ be a real reflexive Banach space.
\begin{theorem}\label{Rthm}
Let $S\colon\ E \toto E^*$ be closed and monotone.   Then $S$ is quasidense if, and only if, for all $x \in E$, $_xS + J$ is surjective. 
\end{theorem}
\begin{proof}
``If'' was established in \Lem{Jlem}.   Suppose, conversely, that $S$ is quasidense and $x \in E$.   Since $G(_xS) = G(S) - (x,0)$, $_xS$ is closed, monotone and quasidense.  \Thm{RLMAXthm} implies that $_xS$ is maximally monotone, and the surjectivity of $_xS + J$ follows from \cite[Theorem 10.7, p.\ 24]{MANDM}.
\end{proof}
\begin{remark}\label{Rrem}
If the norm of $E$ is produced by an Asplund renorming, one can use Rockafellar's generalization \cite{RTRSUMS} of Minty's theorem instead of the result cited from \cite{MANDM} to prove that $_xS + J$ is surjective in \Thm{Rthm}($\lr$).
\par
We shall see in \Ex{Rex} below that the surjectivity of $S + J$ alone is not enough to ensure the quasidensity of $S$ in \Thm{Rthm}($\rl$).
\end{remark}
\begin{example}\label{Rex}
Define the norms $\|\cdot\|_1$ and $\|\cdot\|_\infty$ on $\RR^2$ by $\|(x_1,x_2)\|_1 = |x_1| + |x_2|$ and $\|(y_1,y_2)\|_\infty = |y_1| \vee |y_2|$.   Let $E := (\RR^2,\|\cdot\|_1)$.   Then $E^* = (\RR^2,\|\cdot\|_\infty)$.  Let $A$ be the union of the two axes in $E$, that is to say, $A = (\RR,0) \cup (0,\RR)$.   Define $S\colon E\toto E^*$ by 
\begin{equation*}
S(x) =
\begin{cases}
J(x) &(x \in A);\\
\emptyset &(x \not\in A).
\end{cases}
\end{equation*}
Since $G(S) = G(J) \cap (A \times E^*)$, $S$ is closed and monotone.   We shall prove that
{\em\begin{equation}\label{R1}
S + J \hbox{ is surjective}
\end{equation}
but
\begin{equation}\label{R2}
S\hbox{ is not quasidense.}
\end{equation}}
\par
\noindent
Let $P$ be the square $\{y \in E^*\colon\ \|y\|_\infty = 1\}$,   $P_E$ be the line segment $\{1\} \times [-1,1]$, $P_N$ be the line segment $[-1,1] \times \{1\}$, $P_W$ be the line segment  $-P_E$, and $P_S$ be the line segment $-P_N$.   ($E$, $N$, $W$ and $S$ stand for East, North, West and South, respectively.)   Clearly, $P = P_E \cup P_N \cup P_W \cup P_S$.   Let $e_1 = (1,0)$ and $e_2 = (0,1)$.  
\par
(a)\enspace
If $y \in P_E$, then $\half\|e_1\|_1^2 + \half\|y\|_\infty^2 = \half + \half = 1 =  \bra{e_1}{y}$.    Thus $y \in J(e_1)$.
\par
(b)\enspace
If $y \in P_N$ then, interchanging the indices 1 and 2 in (a), $y \in J(e_2)$.
\par
(c)\enspace
If $y \in P_W$ then $-y \in P_E$.   From (a), $-y \in J(e_1)$, and so $y \in J(-e_1)$.
\par
(d)\enspace
If $y \in P_S$ then $-y \in P_N$.   From (b), $-y \in J(e_2)$, and so $y \in J(-e_2)$.
\par
(e)\enspace Let $V$ be the set consisting of the four points $\pm e_1$ and $\pm e_2$.   It follows from (a)--(d) that $P \subset J(V)$.
\par
(f)\enspace  Let $\lambda > 0$.   From (e), $\lambda P \subset \lambda J(V) = J(\lambda V) \subset J(A)$.   Furthermore, $(0,0) \in J(0,0) \subset J(A)$.   Thus $\RR^2 = \bigcupn_{\lambda > 0}\lambda P \cup \{(0,0)\} \subset J(A)$.
Since $J$ is monotone, so is $S$ and, since $A$ is closed, so is $S$.
\par
(f) shows that $S$ is surjective.  Now $R(S + J) \supset R(S + S) \supset R(2S) = 2R(S)$, and so $S + J$ is surjective, giving \eqref{R1}.    However, since $G(S)$ is a proper subset of $G(J)$, $S$ is not maximally monotone thus, from \Thm{RLMAXthm} not quasidense, giving \eqref{R2}.
\par
This example is patterned after two examples (one due to S. Fitzpatrick and the other due to H. Bauschke)  that appear in \cite[p.\ 25]{MANDM} in which $S$ is monotone, $S + J$ is surjective but $S$ is not maximally monotone.   However, in both of these examples, $S$ is not closed. 
\end{example}
\section{The nonreflexive Banach space case}\label{NRsec}
We now suppose that $E$ is a nonreflexive Banach space, and we discuss a possible analog of \Thm{Rthm}.   \Thm{Rthm}($\rl$) is true in this case too, since the proof does not depend on the reflexivity of $E$ (or even the monotonicity or closedness of $S$).   We shall show in \Ex{NRex} below that \Thm{Rthm}($\lr$) fails in the most spectacular way.
\begin{example}\label{NRex}
Since $E$ is not reflexive, from James's theorem (see Pryce \cite{PRYCE} or Ruiz Gal\'an--Simons \cite{RGS}), there exists $x^* \in E^* \setminus \{0\}$ such that  $x \in E$ and $\|x\| = 1 \lr \bra{x}{x^*} < \|x^*\|$.   It follows that $x \in E \setminus \{0\} \lr \bra{x}{x^*} < \|x\|\|x^*\|$.   We now prove that $x^* \not\in R(J)$.   Indeed, if there existed $x \in E$ such that $x^* \in Jx$ then, from \eqref{J1}, $\half\|x\|^2 + \half\|x^*\|^2 = \bra{x}{x^*} < \|x\|\|x^*\|$, which is manifestly impossible.  Thus $x^* \not\in R(J)$, and so $J$ is not surjective.  Now let $S = 0$.   Then, for all $x \in E$, $_xS = 0$ and so $_xS + J$ is not surjective.   On the other hand, $S$ is (closed, monotone and) quasidense.   The fastest way of seeing this is to note that $S$ is a convex subdifferential and use \Thm{RTRthm}.   However, for the benefit of the reader, we will now give a direct proof of the quasidensity of $S$.   
\par
Let $(x,x^*) \in E \times E^*$ and $\eps > 0$.  The definition of $\|x^*\|$ provides an element $t$ of $E$ such that $\|t\| \le \|x^*\|$ and $\bra{t}{x^*} \ge \|x^*\|^2 - \eps$.   (If $x^* = 0$, we take $t = 0$). Thus, writing $s = t + x$,
\begin{align*}
\half\|s - x\|^2 + \half\|0 - x^*\|^2 + \bra{s - x}{0 - x^*}
&= \half\|t\|^2 + \half\|x^*\|^2 - \bra{t}{x^*}\\
&\le \|x^*\|^2 - \bra{t}{x^*} < \eps.
\end{align*}
Since $(s,0) \in G(S)$, this establishes the quasidensity of $S$.
\end{example}


\end{document}